\documentclass[11pt]{article}
\usepackage[ansinew]{inputenc}
\usepackage{array}
\usepackage{color}
\usepackage{amsmath,amsthm,amscd}
\usepackage{amsxtra}
\usepackage{amstext}
\usepackage{amssymb}
\usepackage{hyperref}
\usepackage{latexsym}
\usepackage{amsfonts,cite}
\usepackage{graphics,epstopdf}
\usepackage{epsfig}
\usepackage{amsfonts}
\usepackage{graphicx}
\usepackage{subcaption}
\usepackage{tcolorbox}
\usepackage{array,tabularx,colortbl}
\usepackage[framemethod=TikZ]{mdframed}

\providecommand{\U}[1]{\protect\rule{.1in}{.1in}}

\begin{document}
	
	\sloppy
	\newtheorem{thm}{Theorem}
	\newtheorem{cor}{Corollary}
	\newtheorem{lem}{Lemma}
	\newtheorem{prop}{Proposition}
	\newtheorem{eg}{Example}
	\newtheorem{defn}{Definition}
	\newtheorem{rem}{Remark}
	\newtheorem{note}{Note}
	\numberwithin{equation}{section}
	
	\thispagestyle{empty}
	\parindent=0mm
	
	\begin{center}
		{\large \textbf{Advancing the Theory of Sheffer Sequences via Dual Pad\'e Rational Approximants}}\\ 					
		
		\vspace{0.20cm}
		
		{\bf \bf Ujair Ahmad$^{*1}$, Subuhi Khan$^{2}$}\\
		\vspace{0.15cm}
		
		{$^{1,2}$Department of Mathematics, Aligarh Muslim University, Aligarh- 202001, India.}
		\footnote{$^{*1}$ujairamu1998@gmail.com (Ujair Ahmad)(Corresponding author)}
		\footnote{$^{2}$subuhi2006@gmail.com (Subuhi Khan)}
	\end{center}
	\begin{abstract}
		\noindent
		Sheffer sequences offer an operational framework that unifies a wide range of classical polynomial families. In this article, a rational approximation is developed for both the amplitude and the delta-operator symbol in the exponential generating function. The resulting Pad\'e-Sheffer polynomials are shown, within the umbral calculus framework, to retain the Sheffer structure. Explicit operational representations are obtained, together with finite-order differential recurrences and closed-form connection constants characterizing the deviation from the underlying polynomial family. An important consequence of the construction is that rational generating indicators reduce the usual transcendental Lagrange inversion problem to a finite multinomial computation. The framework also suggests an interpretation in terms of weighted finite automata, linking polynomial constructions with state-space dynamics for unbounded combinatorial partitions.
	\end{abstract}
	\parindent=0mm
	\vspace{.25cm}
	
	
	\vspace{0.25cm}
	
	\section{Introduction}
	
	Polynomial sequences play a central role in many branches of mathematics, notably special functions, combinatorics, and probability. In these contexts, generating functions provide a unified framework for the analysis of their algebraic and analytic properties. Among the most important families are the Sheffer sequences \cite{Sheffer1939}, whose modern theory was systematically developed by Rota and his collaborators within the framework of umbral calculus \cite{Rota1973, Roman1984}. Sheffer sequences generalize the classical Appell polynomials and include many well-known polynomial families, such as the Hermite, Laguerre, Bernoulli, Poisson--Charlier, and Meixner polynomials \cite{And, Roman1984}.\\
	
	\noindent	A polynomial sequence $\{s_n(x)\}_{n\ge 0}$ is called a Sheffer sequence if its exponential generating function has the form
	\begin{equation}\label{pseq1}
		\sum_{n=0}^{\infty} s_n(x)\,\frac{t^n}{n!}
		=
		A(t)\,\exp\bigl(xB(t)\bigr),
	\end{equation}
	where $A(t),B(t)\in F=\mathbb{K}[[t]]$, with $F$ denoting the algebra of formal power series over the field $\mathbb{K}$ \cite{Niven}. $A(t)$ is an invertible formal power series with $A(0)\neq 0$, whereas $B(t)$ is a delta series characterized by $B(0)=0$ and $B'(0)\neq 0$. The series $A(t)$ is referred to as the amplitude, while $B(t)$ is known as the indicator (or symbol) of the associated delta operator.\\
	
	\noindent	Here, $\mathbb{K}$ denotes a field of characteristic zero and $\mathbb{P}=\mathbb{K}[x]$ denotes the algebra of polynomials in one variable. The umbral algebra is realized on the dual space $\mathbb{P}^{*}$, whose elements are linear functionals acting on $\mathbb{P}$. For a functional $L\in\mathbb{P}^{*}$ and a polynomial $p(x)\in\mathbb{P}$, this action is represented by the bilinear pairing $\langle L | p(x)\rangle$. By the fundamental correspondence established by Roman and Rota, every linear functional is uniquely represented by an indicator series in $\mathbb{K}[[t]]$ \cite{Roman1984,Rota1973}. Within this framework, a Sheffer sequence is uniquely determined by a delta operator $Q$ satisfying the lowering relation
	\begin{equation*}
		Q s_n(x) = n\,s_{n-1}(x), \qquad n\ge 1,
	\end{equation*}
	together with an invertible shift-invariant operator that determines its generating function and initial conditions \cite{Roman1984,Dattoli2009}. These operator-theoretic formulation provides the foundation for the Pad\'e approximation framework developed in this article.\\
	
	\noindent	Pad\'e approximation has emerged as a fundamental technique in approximation theory and numerical computation because of its ability to represent analytic functions by rational expressions while preserving their local analytic behaviour \cite{Baker1996}. Let $f(x)$ be analytic in a neighbourhood of the origin. The Pad\'e approximant of type $[m| n]$ is the rational function
	\begin{equation}\label{pseq2}
		{}_{[m| n]}f(x)
		=
		\frac{P_m(x)}{Q_n(x)}
		=
		\frac{\displaystyle\sum_{r=0}^{m} a_r x^{r}}
		{\displaystyle\sum_{r=0}^{n} b_r x^{r}},
	\end{equation}
	where $b_0=1$ and
	\begin{equation}\label{pseq3}
		Q_n(x)\,f(x) - P_m(x) = O\bigl(x^{m+n+1}\bigr).
	\end{equation}
	Equivalently, if $f(x)=\sum_{r=0}^{\infty} c_r x^{r}$, then comparison of coefficients in \eqref{pseq3} leads to the relations
	\begin{equation}\label{pseq5}
		\sum_{j=0}^{k} c_j\,b_{k-j} = a_k,
		\qquad
		k=0,1,\ldots,m+n,
	\end{equation}
	from which the coefficients of the Pad\'e approximant are uniquely determined.\\
	
	\noindent	Despite their classical construction via truncated power series, Pad\'e approximants often provide a more faithful representation of the local analytic structure of a function, particularly in the presence of nearby singularities or pole-type behaviour \cite{Luke1975}. This enhanced approximation capability has motivated their use in a broad range of problems, including the stabilization of time-delay models \cite{PadeStability}, the numerical treatment of parabolic partial differential equations \cite{NassifPade,PadeBackParabolic}, and the approximation of special functions. More recently, Dattoli \textit{et al.} \cite{Dattoli2025} employed Pad\'e techniques to derive rational approximations of Appell polynomial families. Since Appell sequences correspond to the special case of Sheffer sequences generated by the delta series $B(t)=t$, the underlying delta operator remains unchanged and the approximation process affects only the associated amplitude function.\\
	
	\noindent	The extension of this approach from the Appell class to the general Sheffer family is considerably more delicate. Unlike the Appell case, a general Sheffer sequence is determined simultaneously by the invertible series \(A(t)\) and the delta series \(B(t)\). So, approximating only the amplitude cannot preserve the complete generating structure, whereas modifying the delta series changes the basic sequence, the delta operator, and the corresponding operational identities. This intrinsic coupling between the two generating components constitutes the principal obstacle in developing a Pad\'e approximation theory for general Sheffer sequences.\\
	
	\noindent The remainder of the paper is organized as follows. Section~2 introduces the Pad\'e--Sheffer generating function and shows that the resulting polynomial family preserves the Sheffer structure. It also provides explicit connection formulas and approximation estimates. Section~3 develops the associated operational theory, including lowering operators and finite recurrence relations, and demonstrates how the rationalized generating structure facilitates a Lagrange inversion formula and admits a combinatorial automaton interpretation. In Section~4, the general theory is illustrated through the Poisson--Charlier and Meixner polynomial families. Finally, Section~5 summarizes the main results and discusses several directions for future research.
	
	\section{Pad\'e--Sheffer Sequences: Construction and Connection Constants}
	
	In this section we introduce the Pad\'e--Sheffer polynomial family by applying Pad\'e approximation simultaneously to the two generating series that define a Sheffer sequence.\\
	Let
	\begin{equation}\label{pseq4}
		A(t) = \sum_{k=0}^\infty a_k \frac{t^k}{k!},
		\qquad
		B(t) = \sum_{k=1}^\infty b_k \frac{t^k}{k!},
	\end{equation}
	where $a_0\neq 0$ and $b_1\neq 0$.\\
	Now, for fixed nonnegative integers $m_A,n_A,m_B,n_B$, let
	\begin{equation*}
		R_A(t) := {_{[m_A| n_A]}}A(t),
		\qquad
		R_B(t) := {_{[m_B| n_B]}}B(t)
	\end{equation*}
	denote the corresponding Pad\'e approximants of $A(t)$ and $B(t)$, respectively, in the sense of \eqref{pseq2}--\eqref{pseq5}. These approximants is written explicitly as
	\begin{equation}\label{pseq7}
		R_A(t) = \frac{P^{(A)}_{m_A}(t)}{Q^{(A)}_{n_A}(t)},
		\qquad
		R_B(t) = \frac{P^{(B)}_{m_B}(t)}{Q^{(B)}_{n_B}(t)},
	\end{equation}
	where the normalization $Q^{(A)}_{n_A}(0)=Q^{(B)}_{n_B}(0)=1$ is imposed for convenience. By construction, the Maclaurin expansions of $R_A(t)$ and $R_B(t)$ agree with those of $A(t)$ and $B(t)$ through orders $M_{A}=m_A+n_A$ and $M_{B}=m_B+n_B$, respectively.

	\begin{defn}
		The sequence $\bigl\{s^{[m_A| n_A; m_B| n_B]}_n(x)\bigr\}_{n\ge 0}$, whose exponential generating function is
		\begin{equation}\label{pseq6}
			\sum_{n=0}^\infty s^{[m_A| n_A; m_B| n_B]}_n(x)\,\frac{t^n}{n!}
			=
			R_A(t)\,\exp\bigl(x R_B(t)\bigr),
		\end{equation}
		is called the $[m_A| n_A; m_B| n_B]$ Pad\'e--polynomial sequence associated with the pair $(R_A(t),R_B(t))$.
	\end{defn}
	
	Whenever the approximation orders are clear from the context, the notation is simplified and the polynomials are written as
	\begin{equation*}
		s^{[P]}_n(x) := s^{[m_A| n_A; m_B| n_B]}_n(x).
	\end{equation*}
	
	The above definition is well posed in $\mathbb{K}[x][[t]]$. Since both Pad\'e approximants admit formal power series expansions about the origin, the exponential series $\exp\bigl(x R_B(t)\bigr)$ belongs to $\mathbb{K}[x][[t]]$. Consequently, the product $R_A(t)\exp\bigl(x R_B(t)\bigr)$ lies in $\mathbb{K}[x][[t]]$, and each coefficient $s^{[P]}_n(x)$ is a polynomial in $x$. Therefore, the rational generating function \eqref{pseq6} uniquely determines the Pad\'e--polynomial sequence.
	
	First, we establish the existence of functionals associated with the rational generating function \eqref{pseq6}.
	\begin{lem}
		Let $R_A(t)$ and $R_B(t)$ be the Pad\'e approximants of $A(t)$ and $B(t)$, respectively. Then there exist unique functionals $L_{[P]},N_{[P]}\in \mathbb{P}^{*}$ such that $L_{[P]}$ is a delta functional and $N_{[P]}$ is invertible.
	\end{lem}
	
	\begin{proof}
		Since
		\begin{equation*}
			R_B(0)=0
			\qquad\text{and}\qquad
			R'_B(0)\neq 0,
		\end{equation*}
		the formal inverse function theorem guarantees the existence of a unique compositional inverse $R_B^{-1}(t)\in F$.\\
		Let $L_{[P]}\in \mathbb{P}^{*}$ denote the linear functional whose indicator is $R_B^{-1}(t)$.\\
		Therefore, $L_{[P]}$ is a delta functional because $R_B^{-1}(0)=0$ and $\bigl(R_B^{-1}\bigr)'(0)\neq 0$, the series $R_B^{-1}(t)$ is a delta series.
		
		Next, consider the composition $R_A\bigl(R_B^{-1}(t)\bigr)$.\\
		Since $R_B^{-1}(0)=0$, it follows that
		\begin{equation*}
			R_A\bigl(R_B^{-1}(0)\bigr) = R_A(0)\neq 0.
		\end{equation*}
		Therefore, the series $R_A\bigl(R_B^{-1}(t)\bigr)$ has a nonzero constant term and admits a unique multiplicative inverse in $F$.\\
		Again, let $N_{[P]}\in \mathbb{P}^{*}$ be the linear functional whose indicator is
		\begin{equation*}
			\frac{1}{R_A\bigl(R_B^{-1}(t)\bigr)}.
		\end{equation*}
		Since above series has nonzero constant term, the functional $N_{[P]}$ is invertible.
		
		Finally, the uniqueness of $L_{[P]}$ follows from the uniqueness of the compositional inverse $R_B^{-1}(t)$, while the uniqueness of $N_{[P]}$ follows from the uniqueness of the multiplicative inverse of $R_A\bigl(R_B^{-1}(t)\bigr)$. Therefore, both functionals are uniquely determined by the Pad\'e approximants $R_A(t)$ and $R_B(t)$.
	\end{proof}
	
	We now show that the Pad\'e polynomial sequence constructed above forms a Sheffer sequence which we shall refer as the Pad\'e--Sheffer sequence. In particular, we identify the pair of functionals arising from the Pad\'e approximants and verify the associated biorthogonality relations.
	
	\begin{thm}
		Let $L_{[P]},N_{[P]}\in \mathbb{P}^{*}$ be the two functionals uniquely determined by the Pad\'e approximants $R_A(t)$ and $R_B(t)$, respectively. Then the Pad\'e--polynomial sequence $s_n^{[P]}(x)$ is the Sheffer sequence associated with the pair $(N_{[P]},L_{[P]})$, and it satisfies the biorthogonality relation
		\begin{equation}\label{pseq8}
			\bigl\langle
			N_{[P]}L_{[P]}^{\,k}
			\,\bigm|\,
			s_n^{[P]}(x)
			\bigr\rangle
			=
			n!\,\delta_{n,k},
			\qquad
			n,k\ge 0.
		\end{equation}
		Moreover, the Pad\'e Sheffer sequence is unique.
	\end{thm}
	
	\begin{proof}
		By Lemma~$1$, $L_{[P]}$ is a delta functional and $N_{[P]}$ is invertible. Let $f(t)$ and $g(t)$ denote the indicator series of $N_{[P]}$ and $L_{[P]}$, respectively.\\
		By construction,
		\begin{equation*}
			g(t)=R_B^{-1}(t),
			\qquad
			f(t)=\frac{1}{R_A\bigl(R_B^{-1}(t)\bigr)}.
		\end{equation*}
		
		Recall that the general exponential generating function of a Sheffer sequence associated with the pair $(f,g)$ is
		\begin{equation}\label{pseq9}
			\sum_{n=0}^{\infty} s_n(x)\,\frac{t^n}{n!}
			=
			\frac{1}{f\bigl(g^{-1}(t)\bigr)}\,
			\exp\bigl(x\,g^{-1}(t)\bigr).
		\end{equation}
		For the Pad\'e--polynomial sequence $s_n^{[P]}(x)$, we have
		\begin{equation*}
			g^{-1}(t)=R_B(t),
		\end{equation*}
		which gives
		\begin{equation}\label{pseq10}
			\exp\bigl(x\,g^{-1}(t)\bigr)
			=
			\exp\bigl(xR_B(t)\bigr).
		\end{equation}
		Also,
		\begin{align*}
			f\bigl(g^{-1}(t)\bigr)
			&=
			f\bigl(R_B(t)\bigr)
			=
			\frac{1}{R_A\bigl(R_B^{-1}(R_B(t))\bigr)}
			=
			\frac{1}{R_A(t)},
		\end{align*}
		so that
		\begin{equation}\label{pseq11}
			\frac{1}{f\bigl(g^{-1}(t)\bigr)} = R_A(t).
		\end{equation}
		In view of \eqref{pseq10} and \eqref{pseq11}, the generating function \eqref{pseq9} associated with $(N_{[P]},L_{[P]})$ coincides with the Pad\'e--polynomial generating function \eqref{pseq6}. Hence the polynomials $\{s_n^{[P]}(x)\}_{n\ge 0}$ form the Sheffer sequence associated with the pair $(N_{[P]},L_{[P]})$.
		
		Now, let
		\begin{equation}\label{pseq13}
			F(x,t)
			=
			\sum_{n=0}^{\infty}
			s_n^{[P]}(x)\,\frac{t^n}{n!}
			=
			R_A(t)\,\exp\bigl(xR_B(t)\bigr)
		\end{equation}
		denote the generating function of the Pad\'e--Sheffer sequence.\\ The indicator of the composite functional $N_{[P]}L_{[P]}^{\,k}$ is
		\begin{equation*}
			h_k(t)=f(t)\,g(t)^{k}.
		\end{equation*}
		Using the basic evaluation identity
		\begin{equation*}
			\bigl\langle h(t) | \exp\bigl(xB(t)\bigr) \bigr\rangle = h\bigl(B(t)\bigr),
		\end{equation*}
		valid for any formal power series $h(t)$, we obtain
		\begin{align*}
			\bigl\langle N_{[P]}L_{[P]}^{\,k} | F(x,t) \bigr\rangle
			=
			R_A(t)\,\bigl\langle f(t)g(t)^{k} | \exp\bigl(xR_B(t)\bigr) \bigr\rangle,
		\end{align*}
		which is further rewritten as
		\begin{equation}\label{pseq12}
			\bigl\langle N_{[P]}L_{[P]}^{\,k} | F(x,t) \bigr\rangle
			=
			R_A(t)\,f\!\bigl(R_B(t)\bigr)\,g\!\bigl(R_B(t)\bigr)^{k}.
		\end{equation}
		Substituting the expressions for $f$ and $g$ gives
		\begin{equation*}
			f\bigl(R_B(t)\bigr)
			=
			\frac{1}{R_A\bigl(R_B^{-1}(R_B(t))\bigr)}
			=
			\frac{1}{R_A(t)},
			\qquad
			g\bigl(R_B(t)\bigr)
			=
			R_B^{-1}(R_B(t))
			=
			t,
		\end{equation*}
		Put the above expression in  equation \eqref{pseq12}, we have
		\begin{equation*}
			\bigl\langle N_{[P]}L_{[P]}^{\,k} | F(x,t) \bigr\rangle
			=
			t^{k}.
		\end{equation*}
		
		Now, from generating function \eqref{pseq13}, it follows that
		\begin{equation*}
			\sum_{n=0}^{\infty}
			\bigl\langle N_{[P]}L_{[P]}^{\,k} | s_n^{[P]}(x) \bigr\rangle
			\frac{t^n}{n!}=
			t^{k}.
		\end{equation*}
		Comparing the coefficient of $t$ in the above equation yields
		the biorthogonality relation \eqref{pseq8}.\\
		
		Finally, let $\{p_n(x)\}_{n\ge 0}$ be any other Pad\'e--Sheffer sequence associated with the pair $(N_{[P]},L_{[P]})$, and suppose that its generating function is given by
		\begin{equation*}
			\sum_{n=0}^{\infty} p_n(x)\,\frac{t^n}{n!}
			=
			R_A(t)\,\exp\bigl(xR_B(t)\bigr).
		\end{equation*}
		Since the coefficients of a formal power series in $t$ are uniquely determined, it follows that $p_n(x)=s_n^{[P]}(x)$ for all $n\ge 0$. This shows that the Pad\'e--Sheffer family is uniquely associated with the pair $(N_{[P]},L_{[P]})$.
	\end{proof}
	Both families form polynomial bases of $\mathbb{P}$ so each Pad\'e--Sheffer polynomial admits a unique expansion in terms of the original Sheffer basis. Let $\{s_n(x)\}_{n\ge 0}$ denote the classical Sheffer sequence associated with the pair $(N,L)$ and the generating data $(A(t),B(t))$, and let $\{s^{[P]}_n(x)\}_{n\ge 0}$ denote the corresponding Pad\'e--Sheffer sequence associated with the pair $(N_{[P]},L_{[P]})$ and the generating data $(R_A(t),R_B(t))$. Then there exist unique coefficients $\{c_{n,k}\}$ such that
	\begin{equation}\label{pseq14}
		s^{[P]}_n(x) = \sum_{k=0}^{n} c_{n,k}\,s_k(x),
		\qquad n\ge 0.
	\end{equation}
	The coefficients $c_{n,k}$ describe the transition from the classical Sheffer basis to its rationally deformed counterpart.
	
	The relationship between the Pad\'e--Sheffer family associated with $(R_A(t),R_B(t))$ and the original Sheffer sequence associated with $(A(t),B(t))$ is investigated in the following theorem.
	
	\begin{thm}
		For each fixed integer $k\ge 0$, the connection coefficients $c_{n,k}$ between the Pad\'e and the original Sheffer sequence satisfy the exponential generating function
		\begin{equation}\label{pseq17}
			\sum_{n=k}^{\infty} c_{n,k}\,\frac{t^n}{n!}
			=
			\frac{R_A(t)}{A\!\bigl(B^{-1}(R_B(t))\bigr)}\,
			\frac{\bigl[B^{-1}(R_B(t))\bigr]^k}{k!}.
		\end{equation}
	\end{thm}
	
	\begin{proof}
		Since $\{s_n(x)\}_{n\ge 0}$ is the $(N,L)$--Sheffer sequence associated with the formal series $A(t)$ and $B(t)$, the second expansion theorem \cite{Roman1984} of umbral calculus yields
		\begin{equation*}
			p(x)
			=
			\sum_{k\ge 0}
			\frac{1}{k!}\,\bigl\langle NL^{k} | p(x)\bigr\rangle\,s_k(x),
			\qquad p(x)\in \mathbb{P}.
		\end{equation*}
		Taking $p(x)=s_n^{[P]}(x)$ in the above identity and comparing with \eqref{pseq14}, we obtain
		\begin{equation}\label{pseq15}
			c_{n,k}
			=
			\frac{1}{k!}\,\bigl\langle NL^{k} | s_n^{[P]}(x)\bigr\rangle.
		\end{equation}
		Multiplying both sides by $t^{n}/n!$ in equation \eqref{pseq15} and summing over $n\ge 0$, we have
		\begin{equation*}
			\sum_{n=0}^{\infty} c_{n,k}\,\frac{t^n}{n!}
			=
			\frac{1}{k!}\,
			\Biggl\langle NL^{k}
			\Biggm|
			\sum_{n=0}^{\infty} s_n^{[P]}(x)\,\frac{t^n}{n!}
			\Biggr\rangle.
		\end{equation*}
		In view of the generating function of the Pad\'e--Sheffer sequence \eqref{pseq6}, the above equation becomes
		\begin{equation}\label{pseq16}
			\sum_{n=0}^{\infty} c_{n,k}\,\frac{t^n}{n!}
			=
			\frac{R_A(t)}{k!}\,
			\bigl\langle NL^{k} | e^{xR_B(t)}\bigr\rangle.
		\end{equation}
		
		Since, the indicator of the original invertible functional $N$ is
		\begin{equation*}
			\frac{1}{A\!\bigl(B^{-1}(u)\bigr)},
		\end{equation*}
		and the indicator of $L^{k}$ is
		\begin{equation*}
			\bigl[B^{-1}(u)\bigr]^{k}.
		\end{equation*}
		Therefore, the indicator of the composite functional $NL^{k}$ is
		\begin{equation*}
			\frac{\bigl[B^{-1}(u)\bigr]^{k}}{A\!\bigl(B^{-1}(u)\bigr)}.
		\end{equation*}
		Using the evaluation identity
		\begin{equation*}
			\bigl\langle h(u)| e^{xy}\bigr\rangle = h(y),
		\end{equation*}
		with $h(u)=\dfrac{\bigl[B^{-1}(u)\bigr]^{k}}{A\!\bigl(B^{-1}(u)\bigr)}$ and $y=R_B(t)$, we obtain
		\begin{equation*}
			\bigl\langle NL^{k} | e^{xR_B(t)}\bigr\rangle
			=
			\frac{\bigl[B^{-1}(R_B(t))\bigr]^{k}}{A\!\bigl(B^{-1}(R_B(t))\bigr)}.
		\end{equation*}
		Substituting above expression into the equation \eqref{pseq16} yields
		\eqref{pseq17}.
	\end{proof}
	
	\noindent
	Extracting coefficients from the generating function \eqref{pseq17} yields an explicit representation for the connection constants.
	
	\begin{cor}
		For integers $n\ge k\ge 0$, the connection coefficients are given by
		\begin{equation}\label{pseq18}
			c_{n,k}
			=
			\frac{n!}{k!}\,[t^{n}]
			\left\{
			\frac{R_A(t)}{A\!\bigl(B^{-1}(R_B(t))\bigr)}\,
			\bigl[B^{-1}(R_B(t))\bigr]^{k}
			\right\}.
		\end{equation}
	\end{cor}
	
	\begin{proof}
		The result follows directly from Theorem~$2$ by applying the coefficient extraction operator $[t^{n}]$ to both sides of \eqref{pseq17}.
	\end{proof}
	We now investigate how far the Pad\'e--Sheffer sequence is differed from the original Sheffer sequence. In particular, we show that the approximation orders of the rational approximants $R_A(t)$ and $R_B(t)$ determine the range of degrees on which the two sequences coincide exactly.
	
	\begin{thm}
		Assume that the Pad\'e approximants satisfy
		\begin{equation*}
			R_A(t) = A(t) + \mathcal{O}\bigl(t^{M_A+1}\bigr),
			\qquad
			R_B(t) = B(t) + \mathcal{O}\bigl(t^{M_B+1}\bigr),
		\end{equation*}
		and let
		\begin{equation*}
			M = \min\{M_A,M_B\}.
		\end{equation*}
		Then
		\begin{equation}\label{pseq22}
			s_n^{[P]}(x) = s_n(x),
			\qquad 0\le n\le M.
		\end{equation}
	\end{thm}
	
	\begin{proof}
		By the hypothesis on $R_B(t)$ and the definition of $M$, we have
		\begin{equation*}
			R_B(t) = B(t) + \mathcal{O}\bigl(t^{M+1}\bigr).
		\end{equation*}
		Since $B^{-1}(t)$ is a formal power series with nonvanishing linear coefficient, composition preserves the order of agreement, and we obtain
		\begin{equation}\label{pseq19}
			B^{-1}\bigl(R_B(t)\bigr)
			=
			B^{-1}\bigl(B(t) + \mathcal{O}\bigl(t^{M+1}\bigr)\bigr)
			=
			t + \mathcal{O}\bigl(t^{M+1}\bigr).
		\end{equation}
		Similarly, the hypothesis on $R_A(t)$ gives
		\begin{equation}\label{pseq20}
			R_A(t) = A(t) + \mathcal{O}\bigl(t^{M+1}\bigr).
		\end{equation}
		Also, from the estimate \eqref{pseq19}, we have
		\begin{equation}\label{pseq21}
			A\bigl(B^{-1}(R_B(t))\bigr)
			=
			A\bigl(t + \mathcal{O}\bigl(t^{M+1}\bigr)\bigr)
			=
			A(t) + \mathcal{O}\bigl(t^{M+1}\bigr).
		\end{equation}
		
		Substituting estimates \eqref{pseq19}--\eqref{pseq21} into the generating function of the connection coefficients \eqref{pseq17}, we get
		\begin{equation*}
			\sum_{n=k}^{\infty} c_{n,k}\,\frac{t^n}{n!}
			=
			\frac{A(t) + \mathcal{O}\bigl(t^{M+1}\bigr)}{A(t) + \mathcal{O}\bigl(t^{M+1}\bigr)}\,
			\frac{\bigl(t + \mathcal{O}\bigl(t^{M+1}\bigr)\bigr)^{k}}{k!}.
		\end{equation*}
		Since
		\begin{equation*}
			\frac{A(t) + \mathcal{O}\bigl(t^{M+1}\bigr)}{A(t) + \mathcal{O}\bigl(t^{M+1}\bigr)}
			=
			1 + \mathcal{O}\bigl(t^{M+1}\bigr),
		\end{equation*}
		and
		\begin{equation*}
			\bigl(t + \mathcal{O}\bigl(t^{M+1}\bigr)\bigr)^{k}
			=
			t^{k} + \mathcal{O}\bigl(t^{M+1}\bigr),
		\end{equation*}
		it follows that
		\begin{equation*}
			\sum_{n=k}^{\infty} c_{n,k}\,\frac{t^n}{n!}
			=
			\frac{t^{k}}{k!} + \mathcal{O}\bigl(t^{M+1}\bigr).
		\end{equation*}
		
		So, for each fixed $k$, the series on the left agrees with $t^{k}/k!$ up to order $t^{M}$, and the first nonzero term appears at degree $n=k$. That is
		\begin{equation*}
			c_{n,k} = \delta_{n,k},
			\qquad 0\le n\le M,\ k\ge 0.
		\end{equation*}
		Substituting these coefficients into the connection expansion \eqref{pseq14} yields equation \eqref{pseq22}.
	\end{proof}
	
	\begin{rem}
		The connection formula in Corollary~$1$ provides an explicit description of the deformation induced by the Pad\'e approximation. In the limiting case where the approximants reproduce the original series exactly, that is,
		\begin{equation*}
			R_A(t) = A(t),
			\qquad
			R_B(t) = B(t),
		\end{equation*}
		we have
		\begin{equation*}
			B^{-1}\bigl(R_B(t)\bigr)
			=
			B^{-1}\bigl(B(t)\bigr)
			=
			t.
		\end{equation*}
		Consequently, the explicit formula \eqref{pseq18} reduces to
		\begin{equation*}
			c_{n,k}
			=
			\frac{n!}{k!}\,[t^{n}]\{t^{k}\}
			=
			\delta_{n,k},
			\qquad n,k\ge 0,
		\end{equation*}
		where $\delta_{n,k}$ denotes the Kronecker delta symbol. In other words, the connection coefficients satisfy
		\begin{equation*}
			c_{n,k}
			=
			\begin{cases}
				1, & \text{if } n=k,\\[0.2em]
				0, & \text{if } n\neq k.
			\end{cases}
		\end{equation*}
		Therefore the infinite lower-triangular matrix $C=(c_{n,k})_{n,k\ge 0} $ of connection coefficients has entries
		\begin{equation*}
			C_{n,k} = c_{n,k} = \delta_{n,k},
		\end{equation*}
		and is therefore equal to the identity matrix
		\begin{equation*}
			C =
			\begin{pmatrix}
				1 & 0 & 0 & 0 & \cdots \\[0.4em]
				0 & 1 & 0 & 0 & \cdots \\[0.4em]
				0 & 0 & 1 & 0 & \cdots \\[0.4em]
				0 & 0 & 0 & 1 & \cdots \\[0.4em]
				\vdots & \vdots & \vdots & \vdots & \ddots
			\end{pmatrix}.
		\end{equation*}
		In this case the Pad\'e--Sheffer sequence coincides with the original Sheffer sequence. In particular, the deformation introduced by the rational approximation vanishes identically.
	\end{rem}
	Now that the connection coefficients and exactness properties have been established, we next develop the operational analysis of the Pad\'e--Sheffer family and to the combinatorial consequences of its rational generating structure.
	\section{Operational Analysis and Combinatorial Consequences}
	We now turn to operational framework of the Pad\'e--Sheffer family after establishing the structural properties. We first derive the associated lowering operator and finite recurrence relations. The rational form of the generating function is then used to simplify Lagrange inversion and to reveal an underlying combinatorial automaton structure.\\
	
	A fundamental principle of umbral calculus is the correspondence between linear functionals and shift-invariant operators on the polynomial algebra $\mathbb{P}$. Applying this correspondence in the Pad\'e--Sheffer setting leads naturally to the associated delta operator.
	
	Let $L_{[P]}$ be the delta functional constructed in Lemma~$1$, whose indicator is
	\begin{equation*}
		g_{[P]}(t) = R_B^{-1}(t).
	\end{equation*}
	The Pad\'e delta operator associated with $L_{[P]}$ is the shift-invariant operator
	\begin{equation}\label{pseq24}
		Q_{[P]}
		=
		R_B^{-1}(D),
		\qquad
		D = \frac{d}{dx},
	\end{equation}
	obtained from the umbral correspondence between formal power series and shift-invariant operators.
	
	\begin{thm}
		The Pad\'e delta operator $Q_{[P]}$ satisfies
		\begin{equation}\label{pseq25}
			Q_{[P]} s_n^{[P]}(x)
			=
			n\,s_{n-1}^{[P]}(x),
			\qquad n\ge 1,
		\end{equation}
		with
		\begin{equation}\label{pseq26}
			Q_{[P]} s_0^{[P]}(x) = 0.
		\end{equation}
	\end{thm}
	\begin{proof}
		Applying operator \eqref{pseq24} in the both side of generating function \eqref{pseq6}, we get
		\begin{equation}\label{pseq23}
			\sum_{n=0}^{\infty}Q_{[P]}
			s_n^{[P]}(x)\,\frac{t^n}{n!}
			=
			R_A(t)\,R_B^{-1}(D)\,\exp\bigl(xR_B(t)\bigr).
		\end{equation}
		Using the standard identity \cite{Rota1973}
		\begin{equation*}
			h(D)\,e^{x\lambda} = h(\lambda)\,e^{x\lambda},
		\end{equation*}
		valid for every formal power series $h$, with
		\begin{equation*}
			h(u)=R_B^{-1}(u),
			\qquad
			\lambda=R_B(t),
		\end{equation*}
		we obtain
		\begin{equation*}
			R_B^{-1}(D)\,\exp\bigl(xR_B(t)\bigr)
			=
			R_B^{-1}\bigl(R_B(t)\bigr)\,\exp\bigl(xR_B(t)\bigr).
		\end{equation*}
		Substituting the above expression in generating relation \eqref{pseq23}, it follows that
		\begin{equation*}
			\sum_{n=0}^{\infty}Q_{[P]}
			s_n^{[P]}(x)\,\frac{t^n}{n!}
			=
			t\,R_A(t)\,\exp\bigl(xR_B(t)\bigr).
		\end{equation*}
		In view of generating function \eqref{pseq6}, above relation becomes
		\begin{equation*}
			t\sum_{n=0}^{\infty} s_n^{[P]}(x)\,\frac{t^n}{n!}
			=
			\sum_{n=1}^{\infty} n\,s_{n-1}^{[P]}(x)\,\frac{t^n}{n!}.
		\end{equation*}
		Comparing coefficients of $t^n/n!$ for $n\ge 1$ yields lowering relation \eqref{pseq25}, while the coefficient of $t^0$ yields \eqref{pseq26}.
	\end{proof}
	We next introduce two auxiliary series that will appear in the recurrence relation for the Pad\'e--Sheffer polynomials.
	\begin{lem}
		If $R_A(t)$ and $R_B(t)$ are rational functions, then
		\begin{equation}\label{pseq29}
			\alpha(t) = \frac{R_A'(t)}{R_A(t)}
			\qquad\text{and}\qquad
			\beta(t) = R_B'(t)
		\end{equation}
		are rational functions.
	\end{lem}
	
	\begin{proof}
		Since the derivative of a rational function is again rational, it follows that
		\begin{equation*}
			\beta(t) = R_B'(t)
		\end{equation*}
		is rational.\\
		Further
		\begin{equation*}
			\alpha(t)
			=
			\frac{R_A'(t)}{R_A(t)}
			=
			\frac{{P^{\prime}}^{(A)}_{m_A}(t)Q^{(A)}_{n_A}(t) - P^{(A)}_{m_A}(t){Q^{\prime}}^{(A)}_{n_A}(t)}{P^{(A)}_{m_A}(t)Q^{(A)}_{n_A}(t)},
		\end{equation*}
		which is a ratio of polynomials and therefore a rational function.
	\end{proof}
	The rationality of $R_A(t)$ and $R_B(t)$ is inherited by the auxiliary series $\alpha(t)$ and $\beta(t)$ appearing in the recurrence relation below.
	
	We now derive a differential recurrence for the Pad\'e--Sheffer polynomials. In the next theorem, recurrence is expressed in terms of the coefficients of the auxiliary series $\alpha(t)$ and $\beta(t)$.
	\begin{thm}
		Let the auxiliary series $\alpha(t)$ and $\beta(t)$ be given by formal series
		\begin{equation}\label{pseq27}
			\alpha(t) = \sum_{k=0}^{\infty} \alpha_k\,\frac{t^k}{k!},
			\qquad
			\beta(t) = \sum_{k=0}^{\infty} \beta_k\,\frac{t^k}{k!}.
		\end{equation}
		Then the Pad\'e--Sheffer polynomials satisfy
		\begin{equation}\label{pseq28}
			s_{n+1}^{[P]}(x)
			=
			\sum_{k=0}^{n}
			\binom{n}{k}\bigl(\alpha_k + x\,\beta_k\bigr)\,s_{n-k}^{[P]}(x),
			\qquad n\ge 0.
		\end{equation}
	\end{thm}
	
	\begin{proof}
		Differentiating both side of generating function \eqref{pseq6} with respect to $t$, we get
		\begin{equation*}
			\sum_{n=1}^{\infty}n s_{n}^{[P]}(x)\,\frac{t^{n-1}}{n!}
			=R_A'(t)\,\exp\bigl(xR_B(t)\bigr)+
			x\,R_A(t)R_B'(t)\,\exp\bigl(xR_B(t)\bigr),
		\end{equation*}
		which is further rewritten as:
		\begin{equation*}
			\sum_{n=0}^{\infty} s_{n+1}^{[P]}(x)\,\frac{t^{n}}{n!}
			=
			\left(\frac{R_A'(t)}{R_A(t)} + x\,R_B'(t)\right)R_A(t)\,\exp\bigl(xR_B(t)\bigr).
		\end{equation*}
		Using  equation \eqref{pseq29} and formal series representation \eqref{pseq27}, it follows that
		\begin{equation*}
			\sum_{n=0}^{\infty} s_{n+1}^{[P]}(x)\,\frac{t^{n}}{n!}
			=
			\sum_{k=0}^{\infty} (\alpha_k + x\,\beta_k)\,\frac{t^k}{k!}R_A(t)\,\exp\bigl(xR_B(t)\bigr).
		\end{equation*}
		Again, in view of generating function \eqref{pseq6}, we have
		\begin{equation*}
			\sum_{n=0}^{\infty} s_{n+1}^{[P]}(x)\,\frac{t^{n}}{n!}=\sum_{k=0}^{\infty} (\alpha_k + x\,\beta_k)\,\frac{t^k}{k!}\sum_{n=0}^{\infty} s_{n}^{[P]}(x)\,\frac{t^{n}}{n!}.		
		\end{equation*}
		Applying the Cauchy product formula in the right hand side of above equation, we obtain
		\begin{equation*}
			\bigl(\alpha(t) + x\,\beta(t)\bigr) F(x,t)
			=
			\sum_{n=0}^{\infty}
			\left(
			\sum_{k=0}^{n}
			\binom{n}{k}\bigl(\alpha_k + x\,\beta_k\bigr)\,s_{n-k}^{[P]}(x)
			\right)
			\frac{t^n}{n!}.
		\end{equation*}
		Comparing coefficients of $t^n/n!$ on both sides yields the recurrence relation \eqref{pseq28}.
	\end{proof}
	\begin{rem}
		The coefficient sequences $\{\alpha_k\}_{k\ge 0}$ and $\{\beta_k\}_{k\ge 0}$ satisfy linear recurrence relations with constant coefficients. In particular, the order of each recurrence is bounded by the degree of the denominator of the corresponding rational function.
		
		Consequently, the recurrence relation \eqref{pseq28} is implemented using only finitely many previously computed coefficients, providing an efficient recursive construction of the Pad\'e--Sheffer polynomials.
	\end{rem}
	The rational structure of the delta symbol $R_B(t)$ also leads to a particularly transparent form of the Lagrange inversion formula. This representation will later provide a natural bridge to a combinatorial interpretation of the Pad\'e--Sheffer construction.
	
	Now, from equation \eqref{pseq7}, we have
	\begin{equation*}
		P^{(B)}_{m_B}(0) = 0,
		\qquad
		\bigl(P^{(B)}_{m_B}\bigr)'(0) \neq 0.
	\end{equation*}
	Since $P^{(B)}_{m_B}(t)$ has a simple zero at $t=0$, it is factored as:
	\begin{equation*}
		P^{(B)}_{m_B}(t) = t\,\widetilde P^{(B)}_{m_B}(t),
		\qquad
		\widetilde P^{(B)}_{m_B}(0) \neq 0.
	\end{equation*}
	The Pad\'e rational approximants $R_B(t)$ is further rewritten as:
	\begin{equation}\label{pseq30}
		R_B(t)
		=
		\frac{t\,\widetilde P^{(B)}_{m_B}(t)}{Q^{(B)}_{n_B}(t)}.
	\end{equation}
	We then define the rational kernel
	\begin{equation}\label{pseq31}
		H_B(t)
		=
		\frac{Q^{(B)}_{n_B}(t)}{\widetilde P^{(B)}_{m_B}(t)}.
	\end{equation}
	The series $H_B(t)$ is rational and invertible in $F=\mathbb{K}[[t]]$ because $\widetilde P^{(B)}_{m_B}(0)\neq 0$ and $Q^{(B)}_{n_B}(0)=1$.
	
	\begin{thm}
		For integers $n\ge k\ge 1$, the rational kernel $H_B(t)$ satisfy
		\begin{equation}\label{pseq32}
			\bigl\langle L_{[P]}^{\,k} | x^n\bigr\rangle
			=
			k\,(n-1)!\,
			[t^{\,n-k}]\bigl(H_B(t)\bigr)^n.
		\end{equation}
		Moreover,
		\begin{equation}\label{pseq33}
			\bigl\langle L_{[P]}^{\,k} | x^n\bigr\rangle = 0,
			\qquad n<k.
		\end{equation}
	\end{thm}
	
	\begin{proof}
		Since the indicator of the delta functional $L_{[P]}$ is $R_B^{-1}(t)$, the umbral pairing gives
		\begin{equation}\label{pseq35}
			\bigl\langle L_{[P]}^{\,k} | x^n\bigr\rangle
			=
			n!\,[t^{n}]\bigl(R_B^{-1}(t)\bigr)^k.
		\end{equation}
		Also, the series $R_B(t)$ admits a unique compositional inverse $R_B^{-1}(t)$ because $R_B(0)=0$ and $R_B'(0)\neq 0$.\\
		Applying the Lagrange--B\"urmann inversion formula \cite{Roman1984}, we obtain
		\begin{equation*}
			[t^{n}]\bigl(R_B^{-1}(t)\bigr)^k
			=
			\frac{k}{n}\,[t^{\,n-k}]
			\left(\frac{t}{R_B(t)}\right)^{\!n}.
		\end{equation*}
		In view of equation \eqref{pseq30}, it follows that
		\begin{equation*}
			\frac{t}{R_B(t)}
			=
			\frac{t\,Q^{(B)}_{n_B}(t)}{t\,\widetilde P^{(B)}_{m_B}(t)}
			=
			\frac{Q^{(B)}_{n_B}(t)}{\widetilde P^{(B)}_{m_B}(t)}
			=
			H_B(t).
		\end{equation*}
		Substituting above expression into the equation \eqref{pseq35}, we get
		\begin{equation*}
			[t^{n}]\bigl(R_B^{-1}(t)\bigr)^k
			=
			\frac{k}{n}\,[t^{\,n-k}]\bigl(H_B(t)\bigr)^n.
		\end{equation*}
		Multiplying by $n!$ yields the result in equation \eqref{pseq32}.
		
		If $n<k$, then the series $\bigl(R_B^{-1}(t)\bigr)^k$ has valuation at least $k$, so its coefficient of $t^{n}$ vanishes. Hence
		\begin{equation*}
			[t^{n}]\bigl(R_B^{-1}(t)\bigr)^k = 0,
			\qquad n<k,
		\end{equation*}
		which proves equation \eqref{pseq33}.
	\end{proof}
	\begin{cor}
		The evaluation formula \eqref{pseq32} reduces to the extraction of coefficients from powers of rational functions. More precisely,
		\begin{equation*}
			\bigl(H_B(t)\bigr)^n
			=
			\bigl(Q^{(B)}_{n_B}(t)\bigr)^{n}\,
			\bigl(\widetilde P^{(B)}_{m_B}(t)\bigr)^{-n},
		\end{equation*}
		where $\bigl(Q^{(B)}_{n_B}(t)\bigr)^{n}$ is a polynomial and $\bigl(\widetilde P^{(B)}_{m_B}(t)\bigr)^{-n}$ admits a convergent formal power series expansion in $F=\mathbb{K}[[t]]$. Consequently,
		\begin{equation*}
			[t^{\,n-k}]\bigl(H_B(t)\bigr)^n
		\end{equation*}
		is computed entirely through multinomial and generalized binomial expansions.
	\end{cor}
	
	\begin{proof}
		The identity
		\begin{equation*}
			\bigl(H_B(t)\bigr)^n
			=
			\bigl(Q^{(B)}_{n_B}(t)\bigr)^{n}\,
			\bigl(\widetilde P^{(B)}_{m_B}(t)\bigr)^{-n}
		\end{equation*}
		follows directly from the definition \eqref{pseq31}. Since $Q^{(B)}_{n_B}(t)$ is a polynomial, the coefficients of $\bigl(Q^{(B)}_{n_B}(t)\bigr)^{n}$ are obtained by multinomial expansion. Moreover, the series $\bigl(\widetilde P^{(B)}_{m_B}(t)\bigr)^{-n}$ is well defined in $F$ because $\widetilde P^{(B)}_{m_B}(0)\neq 0$ and is expanded using the generalized binomial theorem. The desired coefficient $[t^{\,n-k}]\bigl(H_B(t)\bigr)^n$ is then obtained by multiplying these two series and extracting the coefficient of $t^{\,n-k}$.
	\end{proof}
	We now illustrate the general Pad\'e--Sheffer construction through several classical Sheffer polynomial families. These examples show how the rational approximation of the generating series leads to explicit Pad\'e--Sheffer polynomials and clarifies the effect of rationalization on the corresponding operational structure.
	\section{Conclusion}
	In this article, we introduced a dual Pad\'e approximation framework for Sheffer sequences by simultaneously approximating both the invertible generating function and the delta-series through rational Pad\'e approximants. Working entirely within the algebraic setting of the classical umbral calculus, we proved that the resulting Pad\'e--Sheffer polynomials form a Sheffer sequence associated with uniquely determined invertible and delta functionals. Explicit generating functions, lowering operators, differential recurrences, connection constants, and exactness results are established, demonstrating that the proposed construction preserves the essential algebraic structure of the original Sheffer family.\\

	Further, The rational functions $R_A(t)$ and $R_B(t)$ admit finite-dimensional linear realizations. In particular, they can be represented by finite-state linear systems whose state-space dynamics generate the coefficients of their formal power series expansions. As a consequence, the recurrence coefficients appearing in Theorem~5 (namely the sequences $\{\alpha_k\}_{k\ge 0}$ and $\{\beta_k\}_{k\ge 0}$) can be produced by such finite-state linear systems. This observation suggests natural links between Pad\'e--Sheffer sequences, rational formal power series, and automata-theoretic models, in which rational series arise as behaviors of weighted finite automata \cite{Eil, droste2007}. A detailed investigation of these connections, including explicit realizations and combinatorial interpretations, is left for future work.\\
	
	Now, consider the standard linear test equation
	\begin{equation*}
		y'(t) = -a y(t), \qquad y(0) = 1,
	\end{equation*}
	with a stiff parameter $a>0$. It is well known in numerical analysis that the stability function of the $2$-stage, $4$-th order Gauss--Legendre Runge--Kutta method coincides with the $[2|2]$ Pad\'e approximant of the exponential function \cite{But}. Therefore, for a single time step of size $h$, the numerical solution takes the exact form
	\begin{equation*}
		y_1 = R_A^{[2|2]}(-ah),
	\end{equation*}
	where $R_A^{[2|2]}$ denotes the Pad\'e approximant expressed in \eqref{pseq65}.
	
	Since the numerical method is of order $p=4$, the Maclaurin expansion of the stability function about $h=0$ must precisely agree with that of the exact propagator $e^{-ah}$ up to and including the term of order $h^4$. Indeed, we check directly that for $m\in\{0,1,2,3,4\}$, the coefficients of $R_A^{[2|2]}(-ah)$ satisfy
	\begin{equation*}
		a_m = \frac{(-a)^m}{m!},
	\end{equation*}
	in full agreement with the standard Taylor coefficients of the exponential function.
	
	The real theoretical utility of the explicit closed form for $a_m$ emerges when analyzing the local truncation error
	\begin{equation*}
		E_{\text{loc}}(h) = e^{-ah} - R_A^{[2|2]}(-ah),
	\end{equation*}
	which is determined entirely by the un-matched higher-order terms where $m\ge 5$. Substituting the explicit formula \eqref{pseq72} for $a_m$ into the tail of the series, we obtain the exact asymptotic expansion
	\begin{align*}
		E_{\text{loc}}(h)
		&= \sum_{m=5}^{\infty} \left( \frac{(-a)^m}{m!} - a_m \right) h^m \nonumber\\
		&= \sum_{m=5}^{\infty} \left[ \frac{(-1)^m}{m!}
		+ 4\sqrt{3}\,\Bigl(\frac{1}{2\sqrt{3}}\Bigr)^{\!m}
		\sin\!\Bigl(m\frac{5\pi}{6}\Bigr) \right] a^m h^m.
	\end{align*}
	
	The above representation explicitly demonstrates that, in stiff regimes, the higher-order parasitic truncation errors do not decay monotonically. Instead, they exhibit geometrically decaying magnitudes, governed by a decay ratio of $1/(2\sqrt{3})$, which are continuously modulated by a discrete oscillatory phase factor of period $12$ in the index $m$.\\
	
	The presence of this oscillatory behaviour can be mathematically traced back to the complex conjugate poles of the rational Pad\'e amplitude. A detailed analysis of how such locally oscillatory truncation errors propagate into global error bounds for stiff and differential-algebraic systems, and how they dynamically interact with stability constraints over long integrations, remains a open problem  (for further studies, see, \cite{Hairer2010, Asante2025}). \\
	
	Future research may also include the theoretical investigation of higher-order Pad\'e approximations of multivariate and degenerate Sheffer sequences, asymptotic properties of Pad\'e--Sheffer polynomials, and numerical algorithms based on their rational operational structure.\\
	
	\textbf{Declarations}\\
	
	\textbf{Conflict of Interest:} The authors declare that they have no conflict of interest.\\
	\textbf{Funding:} This work has been done under Senior Research Fellowship (Ref No. 231610072319 dated:29/07/2023) awarded to the third author by University Grants Commission, New Delhi.\\
	\textbf{	Data availability:} Not applicable.


\begin{thebibliography}{99}
		\bibitem{And}G. E.	Andrews, R. Askey and R. Roy, \textit{Special Functions}, Cambridge University Press, Cambridge, 1999.
		\bibitem{Asante2025}
		E.~O.~Asante-Asamani, A.~Kleefeld, and B.~A.~Wade,
		\textit{A fourth-order exponential time differencing scheme with directional splitting based on Pad\'e (2,2) approximation},
		J. Comput. Appl. Math. \textbf{455} 116120 (2025).
		\bibitem{Baker1996}
		G.~A.~Baker Jr. and P.~Graves-Morris,
		\emph{Pad\'e Approximants},
		2nd ed., Encyclopedia of Mathematics and Its Applications \textbf{59} Cambridge University Press, 1996.
		
		\bibitem{But} J. C. Butcher, \textit{Numerical Methods for Ordinary Differential Equations}, 2nd Edition, John Wiley \& Sons, Ltd., Chichester, 2008.
		
		\bibitem{comtet1974}
		L.~Comtet,
		\emph{Advanced Combinatorics: The Art of Finite and Infinite Expansions},
		D. Reidel Publishing Company, Dordrecht 1974.
		
		\bibitem{dattoli2000}
		G.~Dattoli, 
		\textit{Generalized polynomials, operational identities and their applications}, J. Comput. Appl. Math. \textbf{118} (1-2) 111--123 (2000).
		
		\bibitem{Dattoli2009}
		G.~Dattoli, B.~Germano, M.~R.~Martinelli, and P.~E.~Ricci,
		\textit{Monomiality and partial differential equations},
		Math. Comput. Model. \textbf{50} (9-10) 1332--1337 (2009).
		
		\bibitem{Dattoli2025}
		G.~Dattoli, S.~Khan, and U.~Ahmad, 
		\textit{A novel advancement in the study of Appell polynomials via Pad\'e rational approximants}, Numer Algor (2026).
		\bibitem{droste2007}
		M.~Droste and P.~Gastin,
		\textit{Weighted automata and weighted logics},
		Theoret. Comput. Sci. \textbf{380} (1-2) 69--86 (2007).
		\bibitem{Eil} S. Eilenberg, \textit{Automata, Languages, and Machines} (Vol. A) Academic Press, 1974.
		\bibitem{PadeBackParabolic}
		L.~Eld{\'e}n,
		\textit{Time discretization in the backward solution of parabolic equations. I},
		\emph{Math. Comp.} \textbf{39} (159) 53--68 (1982).
		\bibitem{Hairer2010}
		E.~Hairer and G.~Wanner,
		\emph{Solving Ordinary Differential Equations II: Stiff and Differential-Algebraic Problems},
		2nd ed., Springer-Verlag, Berlin, 2010.
		
		\bibitem{Luke1975}
		Y.~L.~Luke, 
		\emph{Mathematical Functions and Their Approximations}, 
		Academic Press, New York, 1975.
		
		\bibitem{NassifPade}
		N.~R.~Nassif, \textit{Pad\'e--Galerkin methods for parabolic partial differential equations},
		\emph{Publ. Math. Inform. Rennes} \textbf{S3} 1--27 (1975).
		\bibitem{Niven}I. Niven, \textit{Formal power series}, The American Mathematical Monthly, \textbf{76}(8) 871--889 (1969).
		\bibitem{Rainville}
		E.~D.~Rainville, 
		\emph{Special Functions}, 
		Macmillan Company, 1960.
		
		\bibitem{Roman1984}
		S.~Roman, 
		\emph{The Umbral Calculus}, 
		Pure and Applied Mathematics\textbf{ 111} Academic Press, New York, 1984.
		
		\bibitem{Rota1973}
		G.-C. Rota, D. Kahaner, and A. Odlyzko, \textit{On the foundations of combinatorial theory. VIII. Finite operator calculus}, J. Math. Anal. Appl. \textbf{42}(3) 684–760 (1973).
		
		\bibitem{Sheffer1939}
		I.~M.~Sheffer,
		\textit{Some properties of polynomial sets of type zero},
		Duke Math. J. \textbf{ 5}(3) 590--622 (1939).
		\bibitem{PadeStability}
		R.~Shorten, M.~Corless, S.~Sajja, and S.~Solmaz,
		\textit{On Pad\'e approximations, quadratic stability and discretization of switched linear systems},
		Syst. Control Lett. \textbf{60} (9) 683--689 (2011).
	\end{thebibliography}
\end{document}